\documentclass[11pt]{amsart}
\usepackage{amsmath}%\
\usepackage{amsfonts}%
\usepackage{amssymb}%
\usepackage{amsthm}
\usepackage{a4wide}
\usepackage{changes}

\usepackage{xspace}
\usepackage{xcolor}
\usepackage{hyperref}
\usepackage{cleveref}
\usepackage{tikz}
\usepackage{tikz-cd}

\DeclareMathOperator{\GL}{GL}

\newcommand\calT{\mathcal{T}}

\newcommand\calM{\mathcal{M}}
\newcommand\calU{\mathcal{U}}
\newcommand\ANN{\operatorname{Ann}}

\newcommand{\Tgn}[1][g,n]{\calT_{#1}}
\newcommand\Tgnp{\Tgn[g',n']}
\newcommand{\QTgn}[1][g,n]{Q\calT_{#1}}

\newcommand\RR{\mathbb{R}}
\newcommand\CC{\mathbb{C}}
\newcommand{\PQ}[1][X]{\mathbb{P}Q({#1})^*}
\newcommand{\PQgn}[1][g,n]{\mathbb{P}Q\calT_{#1}^*}
\newcommand\Cgn[1][g,n]{\mathcal{C}_{#1}}
\newcommand\Mgn[1][g,n]{\mathcal{M}_{#1}}

\newcommand\Mod{\operatorname{Mod}}

\newcommand\norm[1]{||#1||}
\newcommand\QM{QM}

\newcommand\nb{\mathcal{N}}

\newcommand\TM{Teichm\"uller\xspace}

\newcommand\muq[1][q]{\norm{#1}\dfrac{\overline{#1}}{|#1|}}
\newcommand\muqu[1][q]{\dfrac{\overline{#1}}{|#1|}}
\newcommand\tmuqu[1][q]{\tfrac{\overline{#1}}{|#1|}}

\newcommand{\tops}{\Sigma}
\newcommand\tr{\operatorname{tr}}

\newtheorem{theorem}{Theorem}[section]

\newtheorem{corollary}[theorem]{Corollary}

\newtheorem{lemma}[theorem]{Lemma}

\newtheorem{proposition}[theorem]{Proposition}

\theoremstyle{definition}
\newtheorem{definition}[theorem]{Definition}
\theoremstyle{remark}
\newtheorem{remark}[theorem]{Remark}

\newcommand\re{\operatorname{Re}}
\newcommand\nm[2][V]{||#2||_{#1}}
\newcommand\id{\operatorname{id}}

\title{Isometric embeddings of  Teichm\"{u}ller spaces are covering constructions}
\author{Frederik Benirschke}
\author{Carlos A. Serv\'an}
\address{Department of Mathematics, University of Chicago}
\email{benirschke@uchicago.edu}
\email{cmarceloservan@uchicago.edu}

\begin{document}

\maketitle
\begin{abstract}
    Let $h: \Sigma' \to \Sigma$ be a branched covering of topological surfaces. By pulling back complex structures, $h$ induces a holomorphic isometric embedding of \TM spaces $\calT{(\Sigma)} \hookrightarrow \calT{(\Sigma')}$. We show that for $\dim\calT(\Sigma)\geq 2$, all isometric embeddings arise from branched coverings. This generalizes a theorem of Royden~\cite{royden-automorphisms}.
    As a consequence we obtain that totally geodesic submanifolds of $\calT(\Sigma)$, which are isometric to some \TM space, are covering constructions. Another consequence is 
    the classification of locally isometric embeddings
    of moduli spaces of Riemann surfaces.
\end{abstract}
\section{Introduction}
Let $\Tgn=\calT(\Sigma)$ be the \TM space of an $n$-marked topological surface $\Sigma$ of genus $g$, where we assume $3g-3+n>0$. A celebrated theorem of Royden \cite{royden-automorphisms}, combined with results by Earle-Kra \cite{earle-kra}, shows that for $2g+n \geq 5$ the mapping class group, the biholomorphism group of $\Tgn$ and the (orientation-preserving) isometry group for the \TM metric are isomorphic. 
In particular all isometries for the \TM metric are induced by a diffeomorphism of topological surfaces.
% Informally this can be paraphrased as ``Isometries of the \TM metric come from geometry''. 
In this paper we study a generalization of Royden's theorem by relaxing isometries to isometric embeddings, namely, distance-preserving maps for the \TM metric. One source of isometric embeddings is from {\em covering constructions}, which can be described as follows. A point in $(X,\phi)\in\calT(\Sigma)$ is defined by a marked Riemann surface $\phi:\Sigma\to X$,  where $\phi$ is a homeomorphism from a fixed topological surface $\Sigma$ to a Riemann surface $X$. Suppose that \[h:\Sigma'\to \Sigma\] is a branched covering of topological surfaces. Then there exists a unique marked Riemann surface $\phi_h:\Sigma' \to X_h$ such that the composition 
\[
\phi\circ h\circ (\phi_h)^{-1}:X_h\to X
\] is holomorphic.
We call a map \[f:\calT(\Sigma) \to \calT(
\Sigma')\] a {\em covering construction} if there exists a branched covering $h:\Sigma'\to \Sigma$ such that for all $(X,\phi)\in\calT(\Sigma)$ we have \[f(X,\phi:\Sigma \to X) = (X_h,\phi_h:\Sigma'\to X_h).\] 
We will discuss covering constructions in more detail in \Cref{sec-prelim} and \Cref{appendix}.

Our main result is that, except for finitely many exceptions of $g$ and $n$, covering constructions are the only examples of isometric embeddings.

\begin{theorem}[\textbf{Isometric embeddings are geometric}]\label{theorem:geometric_isometries} Let $\Tgn$ and $\Tgn[g',n']$ be \TM spaces and
suppose $2g+n\geq 5$. Then any isometric embedding $f:\Tgn\to\Tgn[g',n']$ is a covering construction, up to pre- or post-composition with an orientation-reversing mapping class.
\end{theorem}

In \cite{antonakoudis-isometric}, Antonakoudis showed that isometric embeddings are either holomorphic or antiholomorphic. Thus by possibly composing with an orientation-reversing mapping class, we can focus on holomorphic, isometric embeddings.
We also have a classification in the exceptional cases $2g+n\leq 4$, although the classification is more subtle.
 The only four possibilities for $g$ and $n$ are $(g,n)=(0,4), (1,1), (1,2), (2,0)$.
The first two are $1$-dimensional \TM spaces, the remaining two are of dimension $2$ and $3$, respectively.
\begin{corollary}[\textbf{Low complexity cases}]\label{cor:low-complexity}
Let $f:\Tgn\to\Tgn[g',n']$ be a holomorphic, isometric embedding of \TM spaces. Assume $2g+n\leq 4$.

 \begin{enumerate}
 \item
 If $\dim\Tgn=1$, then $f$ is a \TM disc.
 \item
If $(g,n)=(1,2)$, then $\iota:\Tgn[1,2]\simeq\Tgn[0,5]$ via the hyperelliptic involution and 
 \[
 f=h\circ \iota
 \]for some covering construction  $h:\Tgn[0,5]\to\Tgn[g',n']$.
 \item  If $(g,n)=(2,0)$, 
 then  $\iota:\Tgn[2,0]\simeq\Tgn[0,6]$ via the hyperelliptic involution and 
 \[
 f=h\circ \iota
 \]for some covering construction  $h:\Tgn[0,6]\to\Tgn[g',n']$.
 
 \end{enumerate}
 
 \end{corollary}

\begin{proof}
The case of $1$-dimensional \TM spaces was already considered by Antonakoudis \cite[Thm 1.1]{antonakoudis-isometric}.
For the remaining  \TM spaces we can compose any isometric embedding with an isomorphism to either $\Tgn[0,5]$ or $\Tgn[0,6]$, which both satisfy the assumptions of \Cref{theorem:geometric_isometries}.
\end{proof}

In particular \Cref{theorem:geometric_isometries} and \Cref{cor:low-complexity} classify all local isometric embeddings of moduli spaces $\Mgn\to \Mgn[g',n']$, since any local isometric embedding can be lifted to an isometric embedding of \TM spaces.

\subsection{Motivation via \TM dynamics}
The image of an isometric embedding is an example of  totally geodesic complex submanifold, i.e., a complex submanifold $M\subseteq\Tgn$ such that for any two points $x,y\in M$ the unique geodesic through $x$ and $y$ lies in $M$.
Wright \cite{wright-totally-geodesic} proved that totally geodesic submanifolds of dimension at least $2$ in \TM space  project to subvarieties of $\Mgn$  and  that there are only finitely many totally geodesic subvarieties of dimension at least $2$ in $\Mgn$ for fixed $g,n$.
It was originally conjectured by Mirzakhani that every totally geodesic submanifold of \TM space is the image of a covering construction, but recently in \cite{mcmullen-mukamel-wright,emmw} McMullen, Mukamel, Wright and Eskin discovered examples of totally geodesic submanifolds which are not isometric to any \TM space. Our result about isometric embeddings can be rephrased in this setting as the following.
\begin{corollary}
Let $M\subseteq \Tgn[g',n']$ be a totally geodesic submanifold with $\dim (M)\geq 2$. If $M$ is isometric (as metric space) to some \TM space with the \TM metric, then $M$ is the image of a holomorphic covering construction $f:\Tgn[g,n]\to\Tgn[g',n']$ for  some \TM space $\Tgn$.
\end{corollary}

\begin{proof}
Let $\Tgn$ be a \TM space which is isometric to $M$.
By composing an isometry between $\Tgn$ and $M$  with the embedding of $M$ in $\Tgn[g',n']$, we obtain an isometric embedding $f:\Tgn\hookrightarrow \Tgn[g',n']$ with $f(\Tgn)=M$. After possibly changing the orientation we can assume $f$ is holomorphic. 
If $2g+n\geq 5$, then $f$ is a covering construction by \Cref{theorem:geometric_isometries}.  In the remaining case $2g+n \leq 4$ we conclude from $\dim \Tgn =\dim M  \geq 2$, that either $(g,n)=(1,2)$ or $(g,n)=(2,0)$. We proceed as in the proof of \Cref{cor:low-complexity} and use the isometries $\Tgn[1,2]\simeq \Tgn[0,5]$ and $\Tgn[2,0]\simeq\Tgn[0,6]$ to reduce to the previous case $2g+n\geq 5$.
 \end{proof}

\subsection{Rigidity of isometric embeddings}
Covering constructions  descend to algebraic maps of finite coverings of the moduli space of curves and  thus we obtain a rigidity result for isometric embeddings from  \Cref{theorem:geometric_isometries}.

\begin{corollary}\label{cor-rigid}
    Let $f:\Tgn\to\Tgn[g',n']$ be an isometric embedding and assume $\dim \Tgn \geq 2.$ Then there exists a finite cover $\widetilde{\calM}_{g,n}\to \Mgn[g,n]$ and an algebraic, local isometry $\widetilde{f}:\widetilde{\calM}_{g,n}\to\Mgn[g',n']$ of coarse moduli spaces.
    Conversely, assume that $f:\widetilde{\calM}_{g,n}\to \widetilde{\calM}_{g',n'}$ is a local isometric embedding where $\widetilde{\calM}_{g,n}\to\Mgn,\, \widetilde{\calM}_{g',n'}\to\Mgn[g',n']  $ are finite orbifold covering maps and $\dim \Mgn \geq 2.$
    Then $f$ is induced from a covering construction $\widetilde{f}:\Tgn\to\Tgn[g',n']$ and furthermore 
     the set of local isometric embeddings $f:\widetilde{\calM}_{g,n}\to \widetilde{\calM}_{g',n'}$ is finite. 
    
\end{corollary}

We will give the proof of \Cref{cor-rigid} after the proof of \Cref{theorem:geometric_isometries}.

\subsection*{Related work}
In a different direction Royden's theorem can be generalized  by considering isometric submersions.
An {\em isometric submersion} is a holomorphic map $f:\Tgn\to\Tgn[g',n']$ such that the induced map $df:T\Tgn\to T\Tgn[g',n']$ on tangent spaces maps the unit ball onto the unit ball. In \cite{isometric-submersions} Gekhtman and Greenfield show that, if $g'\geq 1, 2g'+n'\geq5$, every holomorphic isometric submersion $f:\Tgn\to\Tgn[g,',n']$ for the \TM metric is a forgetful map $\Tgn[g,n]\to\Tgn[g,m], n\geq m,$ which forgets some of the marked points.
 
 We also note that there are examples of holomorphic maps of moduli spaces that are neither isometric embeddings nor isometric submersions, for example in \cite{leininger-souto} Aramayona, Leininger and Souto have an example constructed as a composition of a covering construction with a forgetful map, i.e., an isometric embedding followed by an isometric submersion.

\subsection{Outline of the proof of \Cref{theorem:geometric_isometries}}
By \cite{antonakoudis-isometric} every orientation-preserving isometric embedding $f:\calT(\Sigma)\to\calT(\Sigma')$ is holomorphic. Thus from now on we assume $f$ is holomorphic.
To show that $f$ is a covering construction one needs a way to build a branched covering from $f$. At each point $X \in \calT(\Sigma)$ we can identify the cotangent space $T^*_X\calT(\Sigma)$ with the space of integrable quadratic differentials $Q(X)$, and the map $f$ induces a map $f^*:Q(X)\to Q(Y)$ between cotangent spaces, where $Y=f(X)$. If $f$ were a biholomorphism, as in the case of Royden's theorem, then $f^*:Q(Y)\to Q(X)$ is a $\CC$-linear isometry. Royden \cite{royden-automorphisms} then showed that every $\CC$-linear isometry $Q(Y)\to Q(X)$ arises from an isomorphism $X\to Y$.
In \cite{markovic} Markovic gives a new proof of Royden's theorem and shows more generally that every $\CC$-linear isometric embedding $Q(Y)\to Q(X)$ of spaces of quadratic differentials is induced by a branched covering $X\to Y$. It follows that every holomorphic map $f:\calT(\Sigma)\to\calT(\Sigma')$, such that  $f^*$ is an isometric embedding, is a forgetful map, see \cite{isometric-submersions}.

In the case $f$ is a holomorphic, isometric embedding, the induced map $f^*:Q(Y)\to Q(X)$ is surjective instead of injective and we cannot apply Markovic's result directly.
% In particular the natural induced map goes into the wrong direction if one wants to apply
% Markovic's results.
Our workaround is to  define \emph{an umkehr} (backwards) map \[f_!:Q(X)\to Q(Y)\] going in the other direction.
From the definition it will be clear that $f_!$ is continuous and norm preserving. The main technical difficulty is to show that in fact $f_!$ is $\CC$-linear, this is done is \Cref{section-umkehr}. Afterwards we can apply Markovic's theorem and then the proof roughly follows the proof of Royden's theorem given in \cite{markovic}.
We will define the umkehr map in detail in \Cref{section-umkehr} but the
idea is as follows.
A quadratic differential $q\in Q\calT(X)$ in the cotangent space  defines a geodesic $g_q$ in $\calT(\Sigma)$. The isometric embedding $f$ maps geodesics in $\calT(\Sigma)$ to geodesics in $\calT(\Sigma')$. In particular the image $f(g_q)$ is a geodesic, which is generated by a quadratic differential $q'$. The umkehr map $f_!$ sends $q$ to $q'$.

\subsection{Acknowledgements}
We would like to thank Benson Farb for many useful discussions and comments on an earlier draft. We also thank Howard Masur, Alex Wright and Dan Margalit for helpful discussions.

\section{Preliminaries}\label{sec-prelim}

\subsection*{Isometric embeddings and submersions}
Let $(V,\nm{\cdot}), (W,\nm[W]{\cdot})$ be two finite-~dimensional normed vector spaces and $T:V\to W$ be a linear map.
We say $T$ is an {\em isometric embedding} if $\nm{v}=\nm[W]{Tv}$ for all $v\in V$.

We say $T$ is an {\em isometric submersion} if the image of the closed unit ball in $V$ is the closed unit ball in $W$.

The dual of an isometric embedding is an isometric submersion and similarly, the dual of an isometric submersion is an isometric embedding.
An isometric submersion can alternatively be characterized by the property
\[
\nm[W]{v'} = \inf_{v\in V:Tv=v'} \nm{v}.
\] 
In a similar way isometric embeddings and submersion can also be defined for normed vector bundles. Here we recall that a normed vector bundle is a holomorphic vector bundle over a complex manifold with a function, which is $C^1$ outside the zero section and restricts to a norm on each fiber.

\subsection*{The \TM metric} We start by recalling some basic facts about the \TM metric. Fix $g,n$ with $3g-3+n>0$ and let $\Tgn$ be the \TM space of genus $g$ surfaces with $n$ marked points. If we need to specify a base point we write $\Tgn = \calT(\Sigma,z)$, where $\Sigma$ is a compact topological surface and $z\subset \Sigma$ is a finite set of points.
A point in $\Tgn$ is determined by a compact Riemann surface $X$ with a marking $\phi:\Sigma \to X$.
The marked points on $X$ are given by $\phi(z)$.
If it is clear from the context we omit the marking $\phi$ and the marked points $\phi(z)$ and ambiguously write $X\in \Tgn$.
We always work with compact surfaces and marked points rather than punctured surfaces.

Let $d$ be the \TM metric on $\Tgn$. Endowed with the metric $d$, \TM space is a complete metric space such that through any two points there exists a unique geodesic.

The tangent space of $\Tgn$ at $X\in\Tgn$ can be identified with the space of all Beltrami forms, i.e., $(1,-1)$-forms $\mu$ on $X$ with $\nm[\infty]{\mu}<\infty$, modulo infinitesimally trivial Beltrami forms. Dually the cotangent space ${(T^*\Tgn)}_{X}$ at $X$ is identified with the space $Q(X)$ of integrable quadratic differentials on the punctured surface $X - \phi(z)$ or equivalently meromorphic quadratic differentials on $X$ with at most simple poles at the marked points $\phi(z)$ and holomorphic elsewhere.
If there is ambiguity about the marked points we write $Q(X,z)$ instead.
Let $\QTgn$ be the bundle of integrable quadratic differentials over $\Tgn.$
For $q\in Q(X)$, the {\em area norm} $\nm[Q\Tgn]{q}$ is defined by \[
\nm[Q\Tgn]{q}:=\int_X |q|
\] and turns $Q\Tgn$ into a normed vector bundle $(Q\Tgn,\nm[\QTgn]{\cdot})$. 
The pairing between $T\Tgn$ and $T^*\Tgn$ is given by the pairing
\[
\left\langle\mu,q\right\rangle\mapsto \int_X \mu q
\]
between Beltrami forms and quadratic differentials. The infinitesimally trivial Beltrami forms $\mu$ are characterized by the condition
\[
\langle\mu, q \rangle =0 \text{ for all } q\in Q(X).
\]
The dual of the area norm defines a norm $\nm[T\Tgn]{\cdot}$ on the the tangent space to $\Tgn$ by 
\[
\nm[T\Tgn]{\mu} = \sup_{ \{\phi: \nm[Q\Tgn]{\phi}=1\}}\left|\int_{X} \mu\phi\right|.
\]

The norm $\nm[T\Tgn]{\mu}$ can be also computed as 
\[
\nm[T\Tgn]{\mu} = \inf_{\mu'} \nm[\infty]{\mu+\mu'},
\]
where $\mu'$ runs over all infinitesimally trivial Beltrami forms and $\nm[\infty]{\cdot}$ is the essential supremum. For a proof see  \cite[Cor. 6.6.4.1]{hubbard}.

The \TM metric $d$ is obtained by integrating $\nm[T\Tgn]{\cdot}$, i.e., if $\gamma:[0,1]\to \Tgn$ is a \TM geodesic connecting $x,y\in\Tgn$ then \[
d(x,y)=\int_0^1 \nm[\Tgn]{\gamma'(t)}dt.
\]
The norm $\nm[T\Tgn]{\cdot}$ is called  the {\em infinitesimal \TM metric}. Both norms $\nm[T\Tgn]{\cdot}$ and $\nm[Q\Tgn]{\cdot}$ are $C^1$ away from the zero section and  strictly convex, see \cite[9.3 Thm. 3 + Lemma 3]{gardiner}.
Every Beltrami form $\mu$ of norm $1$ can be represented uniquely as $\mu = \tmuqu$ for a quadratic differential $q$ of area $1$. Since the norm $\nm[\QTgn]{\cdot}$ is strictly convex, the differential $q$ can be characterized as the unique quadratic differential of norm $1$ such that
\[
\left\langle\mu,q \right\rangle = 1.
\]

\subsection*{\TM discs}
Let $(X,q)\in Q\Tgn.$
The {\em \TM disc} generated by $q$ is the  holomorphic map $g_q:\Delta\to\Tgn$ given by $t\mapsto X_{\mu_t}$ where $\mu_t=t\muqu$ and $X_{\mu_t}$ is obtained by solving the Beltrami equation for $\mu_t$.
 Results of Antonakoudis and Earle-Kra-Krushkal \cite{antonakoudis-isometric,earle-kra-krushkal} show that every  distance preserving map $\Delta\hookrightarrow\Tgn$ for the Kobayashi metrics is a \TM disc, up to complex conjugation.
 The complex tangent vector of $g_q$  at the origin is  $(g_{q,0})_*\partial_z= \muqu$. We refer to $g_q$ as a complex geodesic. The restriction of $g_q$ to any radial line is a real geodesic for the \TM metric.

 \subsection*{Totally geodesic submanifolds}

 A submanifold $M\subseteq\Tgn$ is called {\em totally geodesic} if for all $x,y\in M$ the unique \TM geodesic between $x$ and $y$ is contained in $M$. 
 The {\em \TM metric} $d_M$ of $M$ is the restriction of the \TM metric of $\Tgn$ to $M$.
 More generally if $M,N\subseteq \Tgn$ are totally geodesic submanifolds and $M\subseteq N$, then we call $M$ a totally geodesic submanifold of $N$.
 % The restriction of the \TM metric to $M$ is the Kobayashi metric for $M$.
Every  geodesic for the \TM metric is the restriction of a holomorphic \TM disc and thus any totally geodesic submanifold is a complex submanifold.

\begin{definition}
Let $M\subseteq \Tgn, \, N\subseteq \Tgn[g',n']$ be totally geodesic submanifolds of \TM spaces.
A map $f:M\to N$ is called an {\em isometric embedding}  if it is distance-preserving for the \TM metrics $d_{M}$ and $d_{N}$, i.e., 
\[
d_{M}(x,y)=d_{N}(f(x),f(y)) \text{ for all } x,y\in M.
\]
\end{definition}
In particular an isometric embedding $f$ sends geodesics to geodesics. By \cite[Thm. 4.3]{antonakoudis-isometric} any  isometric embedding $f$ is either holomorphic or anti-holomorphic. For the rest of this paper we will assume that all isometric embeddings are holomorphic.

\begin{proposition}
Let $f:M\to N$ be a holomorphic, isometric embedding of totally geodesic submanifolds of \TM spaces. Then $f$ is an immersion and $f(M)\subseteq N$ is a totally geodesic submanifold.
\end{proposition} 
\begin{proof}
For every tangent vector $v\in T_XM$ in the tangent space over $X\in M$ there exists a unique geodesic $g_v$ through $X$ with tangent vector $v$.
The isometric embedding $f$ sends $g_v$ to a geodesic in $N$, which necessarily has  non-zero tangent vector.
Thus $f$ is an immersion and $f(M)$ is a submanifold. It remains to see that $f(M)$ is totally geodesic.
Let $x,y \in f(M)$ and choose preimages $a,b\in M$ with $f(a)=x, f(b)=y$. Since $M$ is totally geodesic the unique geodesic $g_{a,b}$ through $a$ and $b$ is contained in $M$. The image $f(g_{a,b})$ is the unique geodesic through $x$ and $y$.
Here we used again that $f$ sends geodesics to geodesics.
\end{proof}

\subsection*{Covering constructions}
Let $\tops, \tops'$ be  topological surfaces of genus $g$ and $g'$, respectively,
and $h:\tops'\to\tops$ a branched covering.
Let $z\subseteq\tops$ a collection of marked points containing all branch points of $h$ and set $z'= h^{-1}(z)$.
There exists a holomorphic  map

\begin{gather*}
H:\calT(\Sigma,z)\to \calT(\tops',z'),\\ (X,\phi:\tops\to X, \phi(z))\mapsto(X',\phi':\tops'\to X', \phi'(z')),
\end{gather*}
where $X'$ is obtained by pulling back complex structures under the branched covering $h$.

Since \TM maps pull back to \TM maps, $H$ is a holomorphic isometric embedding.
We refer to $H$ as a {\em totally marked covering construction}. A general covering construction is obtained from a totally marked covering constructions by forgetting marked points.

\begin{definition}\label{def-covering}
A {\em covering construction} is a holomorphic map $G:\calT(\Sigma,u)\to \calT(\tops',v')$ fitting in a commutative diagram
\begin{equation}
\label{eq-tm-covering}
\begin{tikzcd}
    \calT(\tops,z)\ar[r,"H"]\ar[d,swap] & \calT(\tops',z')\ar[d]\\
    \calT(\tops,u)\ar[r,"G",swap] & \calT(\tops',v)
\end{tikzcd}
\end{equation}
where the vertical maps are forgetful maps and $H$ is a totally marked covering construction. In particular this requires $u\subseteq z, v\subseteq z'$.
\end{definition}

\begin{remark}
The above definition might not appear as the  most natural one, since depending on the choice of marked points $u\subseteq\Sigma, v\subseteq\Sigma'$ it is not obvious that $G$ is well-defined and even if $G$ is well-defined it might not be an isometric embedding.
In the \Cref{appendix} we discuss sufficient and necessary conditions on the branched covering and the marked points $u\subseteq\Sigma, v\subseteq\Sigma'$ for the map $G$ to be a well-defined isometric embedding.
The results are not necessary for our proof of \Cref{theorem:geometric_isometries} and can be skipped during a first reading.
\end{remark}

\section{The geodesic bundle of totally geodesic submanifolds}

The {\em geodesic bundle} $\QM\subseteq \QTgn$ over a totally geodesic submanifold is the set of all pairs $(X,q)$ with $X\in M$ and $q$ a quadratic differential generating a \TM geodesic completely contained in $M$.

We start by showing that $QM$ is a  topological submanifold.
Let $Q_1\Tgn$ be the bundle of quadratic differentials with norm less than $1$.
The map \[
\begin{split}
E:\,Q_1\Tgn&\to\Tgn \times \Tgn\\
(X,q)&\mapsto (X,X_{\mu}),
\end{split}
\]
 where $\mu=\muq$ and $X_{\mu}$ is obtained by solving the Beltrami equation for $\mu$,
is a homeomorphism by a result of Earle \cite{earle}.
Recall that complex geodesics are parametrized by $t\mapsto X_{\mu_t},$ where  $X_{\mu_t}$ is the surface obtained by solving the Beltrami equation for $\mu_t=t\muqu$. Since $M$ is totally geodesic we have 
\begin{equation}\label{eq:emap}
E(Q_1M)=M\times M,
\end{equation}
where $Q_1M=Q_1\Tgn \cap QM$.
Therefore $Q_1M$ is a topological submanifold and thus the same is true for $QM$.

Rees showed that $E$ is not $C^1$ \cite[Thm. 1.]{rees}, therefore it is \textit{a priori} not clear that $QM$ is even a real submanifold. 
The goal of this section is to show that in fact $QM$ is a holomorphic subbundle of the cotangent space of $\Tgn$.  

\begin{proposition} \label{proper-lin-bundle}
Let $M\subseteq\Tgn$ be a totally geodesic  submanifold of dimension $d$. Then the geodesic bundle $QM\subseteq \QTgn$ is a holomorphic subbundle of rank $d$. 
\end{proposition}

\begin{proof}
As a first step we show that $\QM$ is an analytic subvariety of $\QTgn$.
Let $N$ be the projection of $M$ to the moduli space $\Mgn$ and similarly $QN$ the projection of $QM$ to $Q\Mgn$.  Wright \cite[Thm. 1.1]{wright-totally-geodesic} proved that $N$ is an algebraic subvariety and $QN$ is closed and $\GL(2,\RR)$-invariant. Thus, by Filip's theorem on algebraicity of $\GL(2,\RR)$-invariant subsets  \cite[Thm 1.1]{filip-algebraicity}, it follows that the intersection of $QN$ with any stratum $Q\Mgn(\mu)$ is an algebraic variety. We conclude that $QN$ is closed and constructible, hence an algebraic variety. Now $QM$ is an irreducible component of the preimage of $QN$ in $\Tgn$ and therefore an analytic subvariety.

It remains to show that $\QM$ is a linear subbundle of $Q\Tgn$. Let $T^*M$ be the cotangent bundle of $M$ and \[\iota:QM\to {\QTgn}_{|M}\to T^*M\]
be the composition of the inclusion with the pullback on cotangent bundles.
In particular $\iota$ is holomorphic. We claim that $\iota$ is injective and proper.

To show the claim we will first show that $\iota$ is norm-preserving, where we endow $T^*M$ with the quotient norm $\nm[T^*M]{\cdot}$.

Let $V= \ker({\QTgn}_{|M}\to T^*M)$. Fix $q\in QM$ and let $q'=\iota(q)$. By definition of the quotient norm we have
\[
\nm[T^*M]{q'} = \inf_{v \in V} \norm{q+v}.
\]
Since  $\dim(V)<\infty$, the infimum is achieved and since the norm $\nm[\QTgn]{\cdot}$ on $Q\Tgn$ is strictly convex the minimum is unique.
It thus suffices to show that $q$ is a critical point for the norm restricted to $q+V.$
From now on assume $q\neq 0$.  Royden \cite{royden-automorphisms} showed that the norm is differentiable on $q+V$ and if we set $g(t)=\nm[\QTgn]{q+tv}$ for some $v\in V$, then 
\[
g'(0)= \re \int_X v\muqu.
\]
Thus $q$ is a critical point if $\re \int_X v\muqu$ for all $v\in V$. Since $V$ is a complex vector space this is equivalent to 
\[
 \left\langle \tmuqu,v\right\rangle =\int_X v\tmuqu = 0 \text{ for all } v\in V.
\]
Therefore $q\in \QTgn$ is a critical point for the norm on $q+V$ if and only if $\tmuqu \in \ANN(V)$, where $\ANN(V)$ is the annihilator of $V$ under the pairing $\langle \cdot,\cdot\rangle : T\Tgn \times \QTgn \to \CC$.

To conclude the proof that $\iota$ is norm preserving it thus suffices to show  $\tmuqu\in\ANN(V)$ for all $q\in QM$. Let  $q\in QM$ and let $i:TM\hookrightarrow T\Tgn$ be the inclusion of tangent bundles. Then the tangent vector $\tmuqu$ is contained in $TM$ and $\iota$ is dual to $i$, thus 
\[
\langle \tmuqu,v\rangle = \langle i(\tmuqu),v\rangle = \langle \tmuqu,\iota(v)\rangle =0 \text{ for all } v\in V.
\]
% Thus if $q\in QM$, then $\tmuqu\in\ANN(V)$. This finishes the proof that $\iota$ is norm preserving.

 Once we know that $\iota$ is norm preserving it follows that $\iota$ is injective, since if $q'=\iota(q)$ then $q$ is the unique norm minimizer in $q+V$. To see that  $\iota$ is also proper, take any sequence $(X_n,q_n)$ in $QM$ leaving any compact set. In particular $\nm[QM]{q_n}\to\infty$ as $n\to\infty.$
 Since $\iota$ is norm-preserving it follows that $\iota(X_n,q_n)$ also leaves every compact set.

 Thus the map $\iota:QM\to T^*M$ is holomorphic, injective and proper. Both $QM$ and $T^*M$ are topological submanifolds of the same dimension  and thus  $\iota$ is open  by invariance of domain.
Since $M$ is connected, the same is true for $T^*M$. Thus $\iota$ is surjective and therefore a homeomorphism.

So far we have seen that $\iota$ is a holomorphic bijection.
The next goal is to show that indeed $
\iota$ is a biholomorphism.
If it was already known that $QM$ and $T^*M$ are complex manifolds this would follow directly, since every holomorphic bijection between complex manifolds is biholomorphic. \textit{A priori} $QM$ has singularities; let $U,Z\subseteq QM$ be the smooth locus and singular locus, respectively.
Then $\iota_{|U}$ is a holomorphic bijection of complex manifolds and thus a biholomorphism. Set $U'= \iota(U), Z'=\iota(Z)$. The proper mapping theorem implies that $Z'$ is a proper analytic subvariety of $T^*M$  and $\iota^{-1}_{|U'}$ is holomorphic. Thus by Riemann's removable singularity theorem the inverse $\iota^{-1}$ is holomorphic on $T^*M$.
Note also that $\iota^{-1}$ is homogenous, since $\iota$ is.
It follows that fiberwise $\iota^{-1}_{X}: T^*M_X\to {(\QTgn)}_{X}$ is homogenous and differentiable at the origin and therefore linear.
\end{proof}

The geodesic bundle $QM$ is naturally a complex Finsler vector bundle by restricting the area norm  from $\QTgn.$
\begin{corollary} 
The normed  vector bundle $(QM,\nm[\QM]{\cdot})$ is linearly isometric to the cotangent bundle $(T^*M,\nm[T^*M]{\cdot})$ with the quotient norm
and furthermore provides a splitting 
\[
\QM \oplus \nb^*M = {\QTgn}_{|M}
\]
as normed bundles. Here $\nb^*M$ is the conormal bundle of $M$, i.e., the dual of the normal bundle, endowed with the restriction of the area norm on $\QTgn$.
\end{corollary}

\begin{proof}
The isomorphism between $QM$ and $T^*M$ is given by the composition $\iota:QM\to{\QTgn}_{|M}\to T^*M$.
For the second claim we  need to show that $QM$ and $\nb^*M$ intersect only in the zero section. This follows from the description $\nb^*M= \ker({\QTgn}_{|M}\to T^*M)$.
\end{proof}

\begin{remark}[Bilinearity of totally geodesic submanifolds]
A surprising consequence of \Cref{proper-lin-bundle} is that the geodesic bundle $\QM$ of a totally geodesic submanifold is linear in two different ways.
On one hand the fibers of the projection $\QM\to M$ are linear subspaces of the cotangent bundle and on the other hand the intersection of $\QM$ with a stratum coincides with a linear subspace in period coordinates of the stratum.
It seems interesting to investigate whether this ``bilinearity'' characterizes totally geodesic submanifolds.
\end{remark}

\section{The geodesic umkehr map}
\label{section-umkehr}
Let $f:M\to N$ be a holomorphic isometric embedding of totally geodesic submanifolds of \TM spaces. The derivative of $f$ induces a map on tangent bundles $f_*:TM\to f^*TN$ and dually also on cotangent bundles $f^*:f^*T^*N\to T^*M$.
By identifying the cotangent bundle with the geodesic bundle we obtain a map $f^*:f^*QN\to QM$.
\begin{definition}
    
The  {\em geodesic umkehr (backwards) map} 
\[
f_!:QM\to f^*QN
\]
is defined 
as follows. Let $(X,q)\in QM$ be a quadratic differential of area  $1$. Then $q$ defines a \TM geodesic $g$ through $X$ with tangent vector $\muqu$. Since $f$ is an isometric embedding $f(g)$ is a geodesic through $f(X)$ with tangent vector $f_*\muqu=\muqu[q']$ for a unique quadratic differential $q'$ of unit area. Define \[
f_!(X,q)= (f(X),q').
\]Extend $f_!$ to a homogeneous function on all of $QM$ by defining
\[
f_!(X,q) = ||q||_{QM} f_!(X,\muqu).
\]
\end{definition}

\begin{remark}
    The umkehr map is most natural in the  case of a covering construction \[f:\calT(\Sigma)\to\calT(\Sigma')\]  arising from a branched covering \[h:\Sigma'\to \Sigma\] of degree $d$.
    For $X\in \calT(\Sigma), Y=f(X)\in\calT(\tops')$, the dual of the derivative of $f$ induces a pullback $f^*:Q(Y)\to Q(X)$, which is given by the trace map \[\operatorname{tr_h}:Q(Y)\to Q(X).\]
    Recall that the trace map is defined by 
    \[
    \tr_h(q)(x) = \sum_{y\in h^{-1}(x)} q(y),
    \]
    if $x\in X$ is not a branched point and extended continously to branch points.  The umkehr map $f_!:Q(\tops)\to Q(\tops')$ on the other hand is given by pullback of differentials $f_!=\dfrac{1}{d}h^*$.
    Note that  in particular $f_!$ is linear, isometric for the area norm on quadratic differentials and a right inverse to $f^*$. One of our main technical theorems is that the same conclusions hold for any isometric embedding of totally geodesic submanifolds of \TM spaces.

\end{remark}

Recall that there is a homeomorphism \[
\begin{split} \psi_M:QM&\to TM, \\
(X,q)&\mapsto (X,\mu_q),
\end{split}
\]
where
\[
\mu_q =\begin{cases} \muq & \text{ if $q\neq 0$},\\
0 & \text{ otherwise. }
\end{cases}
\]
The map $\psi_N:QN\to TN$ is defined analogously. Both maps fit with $f_!$ into the following commutative diagram
\begin{equation}
    \label{diag-umkehr}
\begin{tikzcd} 
QM\ar[d,"\psi_M",swap]\ar[r,"f_!"] & f^*QN\ar[d,"f^*\psi_N"]\\
TM\ar[r,"f_*",swap] & f^*TN
\end{tikzcd}
\end{equation}

It follows from the above diagram  that 
\[f_! = (f^*\psi_N)^{-1}\circ f_*\circ\psi_M
\] is a composition of continuous maps and thus continuous. The main result of this section is that $f_!$ is indeed much better behaved.

\begin{proposition}[Regularity of $f_!$]\label{linearity_umkehr}
Let $M\subset\Tgn,  \, N\subseteq \Tgn[g',n']$ totally geodesic submanifolds and let $f:M\to N$ be a holomorphic isometric embedding.
Then the umkehr map \[f_!:QM\to f^*QN\] has the following properties.
\begin{enumerate}
    \item The map $f_!$ is a holomorphic, $\CC$-linear embedding of vector bundles,
    \item an isometry for the area norms on $QM$ and $QN$ and 
    \item a right inverse of $f^*:f^*QN\to QM$.
\end{enumerate}

Furthermore the formation of $f_!$ is functorial for isometric embeddings of totally geodesic submanifolds.
\end{proposition}

\begin{proof} It suffices to deal with the case $N=f(M)$. Note that $f_!QM=QN$ and $f^*:f^*QN\to QM$ is holomorphic.
    We will first show that $ f^*\circ f_!=\id$, i.e., $f_!$ is a right inverse of $f^*$.

Recall that the pairing between Beltrami forms in $TM$ and quadratic differentials in $QM$ is given by $(\mu,q)= \int_X \mu q$.
Also recall that every Beltrami form  can be represented uniquely in its equivalence class as $\mu=||\mu||_{\infty}\dfrac{\bar{\phi}}{|\phi|}$.
Let $(X,q)\in QM$ and $k=\nm[QM]{q}$.
Since the \TM metric is dual to the area norm on quadratic differentials we have
\[
k=\nm[TM]{k\dfrac{\bar{q}}{|q|}} =k \sup_{\nm[]{\phi}=1} \left\lvert \int_X \dfrac{\bar{q}}{|q|}\phi \right\rvert 
\]
and as the norm $\nm[TM]{\cdot}$ is strictly convex, there is a unique quadratic differential $\phi$ of norm $1$
such that 
\[
k=\nm[TM]{k\dfrac{\bar{q}}{|q|}} =   \int_X k\muqu\phi.
\]
In fact $\phi=\dfrac{q}{\nm[QM]{q}}$, since 
\[
\int_X k \muqu\cdot\dfrac{q}{\nm[QM]{q}} = \dfrac{k}{\nm[QM]{q}} \int_X |q| = k
\]

First observe that $f^*f_!q\neq 0$ unless $q=0$. This follows from 
\[
\int_X \dfrac{\bar{q}}{|q|}f^*f_!q = \int_X f_*\left( \muqu \right)f_!q = \int_X   \muqu[f_!q]f_!q = \nm[QN]{f_!q} = \nm[QM]{q}. 
\]

Furthermore  \[
\nm[QM]{f^*f_!q}= \inf_{\{q':f^*q'= f^*f_!q\}} \nm[QN]{q'}\leq \nm[QN]{f_!q}= \nm[QM]{q},\]
where the first equality follows since $f^*$ is an isometric submersion.

For $q\neq 0$ we now compute
\[
\begin{split}
\int_X \dfrac{\bar{q}}{|q|} \cdot\dfrac{f^*f_!q}{\nm[QM]{f^*f_!q}}&=\int_X f_*\left(\muqu\right)\cdot  \dfrac{f_!q}{\nm[QM]{f^*f_!q}} \\
&=  \int_X \muqu[f_!q] \cdot \dfrac{f_!q}{\nm[QM]{f^*f_!q}} \\
&= \dfrac{\nm[QN]{f_!q}}{\nm[QM]{f^*f_!q}}\geq 1.
\end{split}
\]
On the other hand \[
\left |\int_X \muqu \cdot\dfrac{f^*f_!q}{\nm[QM]{f^*f_!q}}\right| \leq ||\muqu||_{\infty} \int_X \dfrac{|f^*f_!q|}{\nm[QM]{f^*f_!q}} = 1
\]
we conclude  $f^*f_!q = q$.
This finishes the proof of $(1)$.

By definition $f_!:QM\to f^*QN$ is norm-preserving, hence injective and proper.
Additionally  \[\dim QN = 2\dim N = 2\dim M =  \dim \QM,\]
which follows from \cref{eq:emap}. Now both $QM$ and $QN$ are connected, complex manifolds of the same dimension, thus by invariance of domain  $f_!:QM\to f^*QN$ and $f^*:f^*QN \to QM$ are bijective and inverses of each other.
 Since $f^*$ is linear the same is true for $f_!$. 

Finally, the functoriality of $f_!$ follows from \cref{diag-umkehr} and the functoriality of $f_*$.

\end{proof}

\begin{remark}
Using \cref{diag-umkehr} we can define the umkehr map more generally for any holomorphic map $f:\Tgn\to\Tgn[g',n']$ and from the diagram it follows that $f_!$ is functorial for holomorphic maps.  In  case $f:M\to N$ is an isometric submersion of totally geodesic submanifolds we can also give a geometric description for $f_!$ as follows.
Since $f_*$ is an isometric submersion the pullback map
\[
f^*:f^*QN \to QM
\]
is an isometric immersion and we have the relation
\[
f_!f^*q = q.
\]
To see this note that it suffices to show 
\[
f_*\left(\muqu[f^*q]\right) =\muqu.
\]
Suppose that $\nm[QN]{q}=1$. Then
\[
\int f_*\left(\muqu[f^*q]\right) q = \nm[QM]{f^*q} =\nm[QN]{q}=1 
\]
and since $f_*$ is an isometric  submersion 
\[ \nm[TN]{f_*\muqu}\leq \nm[TM]{\muqu}=1.\]
We thus conclude $f_!f^*q=q.$

\end{remark}

\section{Isometric embeddings are geometric}
In this section, using ideas from Gekhtman and Greenfield~\cite[Section 3]{isometric-submersions}, we complete the proof of \Cref{theorem:geometric_isometries}.
The link between the linearity of the geodesic umkehr map $f_!$ and a covering construction is
provided by the following theorem of \cite[Theorem 1.3]{isometric-submersions}, which relives heavily on ideas coming from a paper of Markovic \cite{markovic}.

\begin{theorem}[{\cite[Theorem 1.3]{isometric-submersions}}]\label{theorem:isometry_to_covering}
    Let $X$ and $Y$ be compact marked Riemann surfaces.  Assume $2g(X) + n(X) \geq 5$, where $n(X)$ is the number of marked points. Let $T:Q(X) \to Q(Y)$ be a $\CC$-linear isometric embedding. Then there is a holomorphic map $h:Y \to X$, and a constant $c \in \CC, |c|=1$ such that $T=c\cdot \dfrac{h^*}{\deg(h)}$. 
\end{theorem}
\begin{remark}
In fact, the proof of \Cref{theorem:isometry_to_covering} gives an explicit description of the map $h$ as follows. Let $\Psi_X:X \to \PQ[X]$ and 
$\Psi_Y:Y \to \PQ[Y]$ be the bicanonical embeddings for $X$ and $Y$.

Then $h$ fits into the diagram:
\begin{center}
    \begin{tikzcd}
    \PQ[Y]  \ar[r, "T^*"] & \PQ[X] \\
    Y \ar[u, "\Psi_Y"] \ar[r, "h"] & X \ar[u,"\Psi_X"']
    \end{tikzcd}    
\end{center}
We emphasize here that $Q(X)$ denotes the space of meromorphic quadratic differentials with at most simple poles at the marked points of $X$, so that $\Psi_X$ might differ from the usual bicanonical embedding in algebraic geometry. The condition $2g(X)+n(X)\geq 5$ guarantees that $\Psi_X$ is an embedding.
\end{remark}

We now apply \Cref{theorem:isometry_to_covering} in the situation where $f:\Tgn\to\Tgnp$ is a holomorphic isometric embedding.
Let $\Cgn[g,n]$ and $\Cgn[g',n']$ be the universal curves over $\Tgn$ and $\Tgn[g',n']$. Applying \Cref{theorem:isometry_to_covering} fiberwise to the umkehr map $f_!: \QTgn \to f^*\QTgn[g',n']$, gives a map 
$H:f^*\Cgn[g',n'] \to \Cgn$, which fits into the commutative diagram,
\begin{center}
\begin{tikzcd}
f^*\Cgn[g',n'] \ar[d] \ar[r, "H"] & \Cgn  \ar[d] \\
f(\Tgn)& \ar[l, "f"'] \Tgn  
\end{tikzcd}
\end{center}
As we will see below, the regularity of $H$ plays a key role in the proof of \Cref{theorem:geometric_isometries}. The following shows that $H$ is holomorphic. 
\begin{lemma}\label{lemma-Hhol}Suppose $2g+n\geq 5$. Then the map $H:f^*\Cgn[g',n'] \to \Cgn$ is holomorphic.
\end{lemma}
\begin{proof}
Consider the following commutative diagram,
\begin{center}
    \begin{tikzcd}
       \Cgn \ar[r,"\Psi"] \ar[rd] & \PQgn \ar[d] & \ar[l,"f_!^*"'] f^*\PQgn[g',n'] \ar[d] & \ar[l, "\Phi"'] \ar[ld] f^*\Cgn[g',n']\\
       &\Tgn \ar[r, "f"] & f(\Tgn) 
    \end{tikzcd}
\end{center}
The maps $\Psi$ and $\Phi$ are the respective fiberwise bicanonical embeddings, which are biholomorphic  onto their image. Since $H$ is given by \[ H = \Psi^{-1} \circ f_!^* \circ \Phi, \] the claim follows.
\end{proof}

\begin{proof}[Finishing the proof of \Cref{theorem:geometric_isometries}]

For the rest of the proof we identify $\Tgn=\calT(\Sigma,u)$ with its image $f(\Tgn) \subset \Tgnp=\calT(\Sigma',v)$. By \Cref{lemma-Hhol}, there exists a holomorphic family of branched coverings 
\[
\begin{tikzcd}
f^*\Cgn[g',n'] \ar[d,"H"] \\
\Cgn \ar[d]\\
\Tgn.
\end{tikzcd}
\]
For a marked surface $(X,u,\phi:\Sigma\to X)\in \Tgn=\calT(\Sigma,u)$ we let  $h_X: f(X) \to X$ be map induced by restricting $H$ to a fiber.
The ramification multiplicity $(X,x) \to \operatorname{mult}_x h_X$ is an upper semicontinuous function on $f^*\Cgn$.
Thus there exists a dense open set $\calU \subset\Tgn$ where the \emph{ramification profile}, i.e., the number of  ramification points and ramification multiplicities, of $h_X$ is constant. Moreover the ramification points neither collide nor split along paths in $\calU$.

For each $(X,\phi,u)$, let $z(X)$ be the union of $u$ and all branch points of $h_X$. Similarly, let $z'(X)$ be the union  of all marked points $v$ on $f(X)$ and all ramification points.
By removing $z$ and $z'$ and varying $X$ we obtain  a holomorphic family of covering maps 
\[
f(X) - z'(X) \to X - z(X).
\]

By passing to the universal cover $\widehat{U}$ of $U$
we obtain marked families and furthermore  the family of coverings can be (smoothly) trivialized by covering space theory. 
Summarizing, there is a diagram (\emph{a priori} not commutative)
\[
\begin{tikzcd}
    \widehat{U} \ar[r]\ar[d] & \calT(\tops,z)\ar[r,"g"] \ar[d,"\pi"]&\calT(\tops',z')\ar[d]\\
    U \ar[r]& \calT(\tops,u) \ar[r,"f",swap] & \calT(\tops',v).
\end{tikzcd}
\]
The first diagram is commutative and the second one is commutative at least over $\pi^{-1}(U)$, which is dense in $\calT(\tops,z)$. Since all maps involved are holomorphic we conclude that the whole diagram commutes.
By construction $g$ is a totally marked covering construction and thus $f$ is a covering construction by definition.

\end{proof}

It remains to address \Cref{cor-rigid}: local isometric embeddings of moduli spaces.

\begin{proof}[Proof of \Cref{cor-rigid}]
Since $\dim \Tgn\geq 2$, \Cref{theorem:geometric_isometries} and \Cref{cor:low-complexity} imply that the isometric embedding $f:\Tgn\to\Tgn[g',n']$ is a covering construction.
Let $\Mod(g,n)=\Mod(\Sigma,u)$ be the mapping class group of $n$-marked genus $g$ surfaces.
Given a branched covering $h: \Sigma'\to \Sigma$, let $\Mod_h\subseteq \Mod(g,n)$ be the finite-index subgroup of mapping classes that lift under the branched covering $h$.
The covering construction then descends to an algebraic map \[ \widetilde{f}: \widetilde{\calM}_{g,n}=\Tgn/ \Mod_h \to \Mgn[g',n']\]  of coarse moduli spaces. 
Note that there is not necessarily an induced group homomorphism $\Mod_h\to \Mod(g',n')$, thus we cannot guarantee a map between orbifolds. Since the map $f$ of \TM spaces is an isometric embedding, the induced map $\widetilde{f}$ is a local isometric embedding.

Conversely, let \[
\widetilde{\calM}_{g,n}= \Tgn/\Gamma, \,\widetilde{\calM}_{g',n'} = \Tgn[g',n']/\Gamma'
\]be finite orbifold covers of $\Mgn$ and $\Mgn[g,'n,']$, respectively, for some finite index subgroups \[
\Gamma\subseteq \Mod(g,n),\, \Gamma'\subseteq \Mod(g',n').\]
Let 
 \[
\widetilde{f}: \widetilde{\calM}_{g,n}\to \widetilde{\calM}_{g',n'}
 \]  be a local isometric embedding. In particular $\widetilde{f}$ is induced by a $\rho$-equivariant local isometric embedding of \TM spaces $f:\Tgn\to\Tgn[g',m']$, where $\rho:\Gamma\to\Gamma'$ is some group homomorphism. Since on $\Tgn$ the distance between two points is realized by a geodesic, $f$ is an isometric embedding and thus a covering construction by \Cref{theorem:geometric_isometries}.
Hence $f$ is obtained by pulling back complex structures along a branched covering $h:\Sigma\to \Sigma'$ of topological surfaces. 
Pre- and post composing with homeomorphisms isotopic to the identity changes the branched covering $h$ but not the induced map $f$ on \TM spaces. For fixed $(g,g',n,n')$ the mapping class groups  $\Mod(g,n)$ and $\Mod(g',n')$ act on
\[
\operatorname{Cov}:= \{ \phi:\Tgn\to\Tgn[g',n']\,|\, \phi \text{ is a  covering construction}\}
\]
by pre- and post-composition, respectively.
Both actions  by $\Mod(g,n)$ and $\Mod(g',n')$ leave the  ramification profile and the monodromy representation of $h$ invariant. 
If a covering construction $\phi$ descends to a map \[\widetilde{\phi}:\widetilde{\calM}_{g,n}\to \widetilde{\calM}_{g',n'},\]
then $\gamma'\phi\gamma$ for $\gamma\in\Gamma,\gamma'\in\Gamma'$ descends to the same map $\widetilde{\phi}$.
In particular there are at most 
\[
[\Mod(g',n'):\Gamma']\cdot  [\Mod(g,n):\Gamma]
\]
coverings constructions 
\[\widetilde{\phi}:\widetilde{\calM}_{g,n}\to \widetilde{\calM}_{g',n'},\]
with  a fixed ramification profile and fixed  monodromy representation.
Since there are only finitely many choices for both ramification and monodromy representation, there are only finitely many coverings constructions from $\widetilde{\calM}_{g,n}$ to $\widetilde{\calM}_{g',n'}$.

\end{proof}

\newpage
\appendix
\section{Covering constructions}\label{appendix}
In \Cref{def-covering}, covering constructions \[G:\calT(\Sigma,u)\to\calT(\Sigma',v)\] associated to a branched covering 
\[
h:\Sigma'\to\Sigma
\]
of topological surfaces were defined. One can always choose the branched points $u\subseteq \Sigma, v\subseteq \Sigma'$ such that $G$ is well-defined. For example if $u$ is the set of branch points of $h$ and $v=h^{-1}(u)$, then $G$ is well-defined and an isometric embedding.
But for other choices of $u$ and $v$ the map $G$ might not be well-defined or not an isometric embedding. In the rest of the appendix we discuss in detail when $G$ is well-defined and if so, when $G$ is an isometric embedding.

A covering construction  $G:\calT(\Sigma,u)\to\calT(\Sigma',v)$ is uniquely determined by the diagram \cref{eq-tm-covering}, i.e., by the data  of a {\em branching tuple} $(h,z,u,v)$ consisting of 
\begin{itemize}
    \item  a branched covering $h: \Sigma'\to \Sigma$, \item finite sets $u\subseteq z \subseteq \Sigma$ such that $z$ contains all branch points of $h$ and
    \item a finite set $v\subseteq z'=h^{-1}(z)\subseteq \Sigma'$.
    \end{itemize}

\begin{definition}
A branching tuple $(h,z,u,v)$ associated to a branched covering $h:\Sigma'\to\Sigma$ is {\em realizable} if there exists a holomorphic map $G$ fitting into a diagram 
\labelcref{eq-tm-covering} where $H$ is a totally marked covering construction.
Similarly, the branching tuple $(h,z,u,v)$ is {\em realizable by an isometric covering construction} if $G$ is an isometric embedding.
\end{definition}

We first address the case where $(h, z,u,v)$ is realizable and discuss  necessary and sufficient conditions for when the resulting covering construction $G$ is an isometric embedding.

\begin{proposition}\label{prop-isometric-covering}
Let $G: \calT(\tops,u)\to  \calT(\tops',v)$ be a covering construction induced by branched covering $h:\Sigma'  \to \Sigma$ and assume $\calT(\tops,u)\not\simeq \Tgn[1,1]$. Then $G$ is an isometric embedding if and only if $u$ contains all branch points of $h$ and  \begin{equation*}
    h^{-1}(u)\subseteq  v \cup R(h),
    \end{equation*}
    where $R(h)\subseteq \Sigma'$ is the set of ramification points of $h$.
   
\end{proposition}

\begin{proof}
Suppose $G$ is an isometric embedding. 
In particular $G$ fits into a diagram 
\begin{equation} \label{eq-covering}
\begin{tikzcd}
    \calT(\tops,z)\ar[r,"H"]\ar[d,swap] & \calT(\tops',h^{-1}(z))\ar[d]\\
    \calT(\tops,u)\ar[r,"G",swap] & \calT(\tops',v) 
\end{tikzcd}
\end{equation}

Given $X\in \calT(\Sigma,z)$ and $X'=H(X)\in\calT(\Sigma',h^{-1}(z))$ we have induced holomorphic branched coverings
\[
h_X: X'\to X
\]
The geodesic  umkehr map $H_!$ is given by pullback of quadratic differentials \[
h_X^*:Q(X,z)\to Q(X',z').
\]
All four maps in the diagram \labelcref{eq-covering} are either isometric immersion or are submersion and thus by the functoriality of the umkehr map we see that $G_!$ is also given by pullback $h_X^*$.
In particular we obtain an induced map
\[
{(h_X^*)}_{|Q(X,u)}:Q(X,u)\to Q(X',v).
\]
Let $p\in u$ be a marked point. Since $(g,n)\neq(1,1)$ we can choose a quadratic differential $q$ on $X$ with a simple pole at $q$ (and potentially more simple poles). The pullback differential $h_X^*(q)$ has a simple pole at every preimage of $p$ unless the preimage is a ramification point.
Since simple poles have to be marked points we conclude that $v$ contains $h^{-1}(u) \setminus R(h)$.

Next we argue that $u$ contains all the branch points $b$. The diagram \labelcref{eq-covering} extends to a diagram
\[
\begin{tikzcd}
    \calT(\tops,z)\ar[r,"H"]\ar[d,swap] & \calT(\tops',h^{-1}(z))\ar[d]\\
    \calT(\tops,u \cup b)\ar[d,"\pi"]\ar[r, "K"] &  \calT(\tops',h^{-1}(u)\cup R(h))\ar[dr,"\pi'"]\\
    \calT(\tops,u)\ar[r,"G",swap] & \calT(\tops',v) \ar[r] & \calT(\tops',h^{-1}(u) - R(h)) 
\end{tikzcd}
\]

The composition $\pi'\circ K$ is realizable and an isometric embedding. In particular $\pi\circ K$ is injective. Since the diagram is commutative, this is only possible if $b\subseteq u$. Otherwise the fibers of $\pi$ would have positive dimension.
This shows the necessity of the condition
\[
b\subseteq u, \quad h^{-1}(u)\subseteq v \cup R(h). 
\]

To see sufficiency note that a \TM map lifts to the branched covering, as long as a quadratic differential with simple poles at $u$
pulls back to a quadratic differential with at most simple poles at $v$, which is guaranteed by $h^{-1}(u)\subseteq v \cup R(h)$.

\end{proof}

Now we address the same question, but without assuming $(h,z,u,v)$ is already realizable.

\begin{proposition}\label{prop-isometric-realizable}
Let $(\tops,u),(\tops',v)$ be topological surfaces with $g(\tops)\geq 1$ or $g(\tops')\geq 2$ and $(g(\tops),|u|)\neq(1,1)$. Furthermore, let $(h,z,u,v)$ be a branching tuple.
Then $(h,z,u,v)$ is realized by an isometric covering construction $G: \calT(\tops,u)\to  \calT(\tops',v)$ 
if and only if \begin{equation*}
    h^{-1}(u) = v \cup  R(h),
    \end{equation*}
    where $R(h)\subseteq \Sigma'$ is the set of ramification points of $h$.
    % In case $(g(\tops),|u|) = (1,1),$ the condition $h^{-1}(u) \supseteq   R(h)$ is a necessary and sufficient condition.   
\end{proposition}

\begin{proof}
First assume that the branching tuple  $(h,z,u,v)$
satisfies
\begin{equation}\label{eq-branching}
h^{-1}(u) = v\cup R(h),
\end{equation}
where $B(h)$ and $R(h)$ are the branching and ramification points of $h$, respectively. In particular $u$ contains all the branch points.
Then $(h,z,u,v)$ is realizable by a covering construction $G$, since it can be obtained from the totally marked covering construction with branching tuple $(h,u,u,h^{-1}(u))$ by forgetting the points in $h^{-1}(u) - v.$
We can then apply \Cref{prop-isometric-covering} to conclude $G$ is an isometric embedding.

It remains to prove the converse and assume that $G$ is an isometric covering construction for some branching tuple $(hz,u,v)$.
By \Cref{prop-isometric-covering} 
it follows 
\[
B(h)\subseteq u,\, h^{-1}(u)\subseteq v\cup R(h).
\]

It thus remains to argue that
$h(v)\subseteq u$, i.e., every marked point of $X'$ lies over a marked point of $X$.
The coderivative map of $G$ can be computed as the trace map \[\tr_{h_X}:Q(X',v)\to Q(X,u)\] by looking at the diagram \cref{eq-covering}. 

If there exists a quadratic differential $q'\in Q(X',v)$ with a simple pole at a marked point $p\in X'$ but no other poles in the fiber $h_X^{-1}(h(p))$, then $\tr_X(q')$ has a simple pole at $h(p)$ and thus $h(p)$ needs to be marked as well.
If $g(X')\geq 2$ there exists a quadratic differential with a simple pole at $p$ and no other poles by Riemann-Roch.

If $g(X') =1$ and $g(X)=1$, then $h_X$ is unramified by Riemann-Hurwitz. Since $u$ has at least two points the same is true for $h^{-1}(u)$. Furthermore $h^{-1}(u)$ is disjoint from $p$, and thus there exists a quadratic differential with a simple pole at $p$ such that the trace has a simple pole at $h(p)$.

\end{proof}

\end{document}